\documentclass{amsart}

\usepackage{color,graphicx,amssymb,verbatim}
\usepackage{caption,subcaption}

\newtheorem{theorem}{Theorem}[section]
\newtheorem{lemma}[theorem]{Lemma}

\theoremstyle{definition}

\theoremstyle{remark}
\newtheorem{remark}[theorem]{Remark}
\numberwithin{equation}{section}

\newcommand{\ph}{\hat\psi}
\newcommand{\eps}{\varepsilon}
\newcommand{\phe}{\ph_{\eps}}

\usepackage{color,graphicx,amssymb}
\usepackage{graphicx}
\allowdisplaybreaks

\begin{document}
\title{Rigorous approximation of diffusion coefficients for expanding maps}
\author{Wael Bahsoun}
    \address{Department of Mathematical Sciences, Loughborough University, 
Loughborough, Leicestershire, LE11 3TU, UK}
\email{W.Bahsoun@lboro.ac.uk}
\author{Stefano Galatolo}
    \address{Dipartimento di Matematica, Universit\`a di Pisa, Largo Pontecorvo, Pisa, Italy}
\email{stefano.galatolo@unipi.it}

\author{Isaia Nisoli}
   \address{Instituto de Matematica - UFRJ Av. Athos da Silveira Ramos 149, Centro de Tecnologia - Bloco C Cidade Universitaria - Ilha do Fund\~ao. Caixa Postal 68530 21941-909 Rio de Janeiro - RJ - Brasil}
\email{nisoli@im.ufrj.br}

\author{Xiaolong Niu}
    \address{Department of Mathematical Sciences, Loughborough University,  Loughborough, Leicestershire, LE11 3TU, UK}
\email{x.niu@lboro.ac.uk}
\subjclass{Primary 37A05, 37E05}
\thanks{WB and SG would like to thank The Leverhulme Trust for supporting mutual research visits through the Network Grant IN-2014-021. SG thanks the Department of Mathematical Sciences at Loughborough University for hospitality. WB thanks Dipartimento di Matematica, Universita di Pisa. The research of SG and IN is partially supported by EU Marie-Curie IRSES ``Brazilian-European partnership in Dynamical Systems" (FP7-PEOPLE-2012-IRSES 318999 BREUDS)}

\date{\today}
\keywords{Transfer Operators, Central Limit Theorem, Diffusion, Ulam's Method, Rigorous Computation.}

\begin{abstract}
We use Ulam's method to provide rigorous approximation of diffusion coefficients for uniformly expanding maps.  An algorithm is provided and its implementation is illustrated using Lanford's map.\end{abstract}
\maketitle
\markboth{Rigorous approximation of diffusion coefficients }{W. Bahsoun, S. Galatolo, I. Nisoli, X. Niu}
\section{Introduction}
The use of computers is essential for predicting and understanding the behaviour of many physical systems. Sensitive dependence on initial conditions is typical in many physical systems. This sensitivity problem raises 
nontrivial reliability and stability issues regarding any computational approach to such systems. Moreover, it strongly motivates the study of reliable computational methods for understanding statistical properties of physical systems.

\bigskip

In this note we consider the rigorous computation of diffusion coefficients
in a class of systems where a central limit theorem holds. Such
coefficients are focal in the study of limit theorems and fluctuations for
dynamical systems (see \cite{D, HK, HM, G, Li, MN} and references therein).
Given a piecewise expanding map, an observable, and a pre-specified
tolerance on error, we approximate in a certified way the diffusion
coefficient up to the per-specified error (see Theorem \ref{main}).

\bigskip

Our rigorous approximation is based on a suitable finite dimensional
approximation (discretization) of the system, called Ulam's method \cite{Ulam}. Ulam's method is known to provide
rigorous approximations of SRB (Sinai-Ruelle-Bowen) measures and other
important dynamical quantities for different types of dynamical systems (see 
\cite{B, BB, BBD, Froy1, Froy2, Li1, GN0, GN, M2, RM} and references
therein). Moreover, this method was also used to detect coherent structures in geophysical systems (see e.g. \cite{Frandal}, \cite{franddel}).

\bigskip

In \cite{Po}, following the approach of \cite{JP}, a Fourier
approximation scheme was used to estimate diffusion coefficients for expanding maps.
The approach of \cite{Po} requires the map to have a Markov partition and to
be piecewise analytic. Although the result of \cite{Po} provides an order of
convergence, it does not compute the constant hiding in the rate of
convergence. In our approach, we do not require the map to admit a Markov
partition and we only assume it is piecewise $C^{2}$. More importantly, our
approximation is rigorous. To give the reader a flavour of what we mean by rigorous, we close this section by providing in part (b) of the following theorem a prototype result of this paper\footnote{Part (a) of Theorem \ref{proto} is well know, see for instance \cite{HK}.  Section \ref{example} contains the application of our method to the Lanford map, which proves Theorem \ref{proto}.}: 
\begin{theorem}\label{proto}
Let\footnote{Computer experiments on the orbit structure of this map were performed by O. E. Lanford III in \cite{Lanford}, and since then it is known as Lanford's map.}
\begin{equation}  
T(x)=2x+\frac{1}{2}x(1-x) \hskip 0.5cm {\text{(mod } 1)}.
\end{equation}
\begin{enumerate}
\item[(a)] $T$ admits a unique absolutely continuous invariant measure $\nu$ and if $\psi$ is a function of bounded variation the Central Limit Theorem holds: 
$$\frac{1}{\sqrt{n}}\left(\sum_{i=0}^{n-1}\psi(T^ix)-n\int_I\psi d\nu\right){\overset{\text{law}}{\longrightarrow}}\mathcal N(0, \sigma^2).$$
\item[(b)] For $\psi=x^2$ the diffusion coefficient $\sigma^2\in [0.3458,0.4152]$.
\end{enumerate}
\end{theorem}

\bigskip

In Section \ref{setting}, we first introduce our framework and the
assumptions on it. We then state the problem and introduce the method of
approximation. The statement of the general results (Theorem \ref{main} and
Theorem \ref{rate}) and an application to expanding maps with a neutral
fixed point are also included in Section \ref{setting}. Section \ref{proofs}
contains the proofs and an algorithm. Section \ref{example} contains an
example, using Lanford's map, that illustrates the implementation of the algorithm of Section \ref{proofs} and proves part (b) of Theorem \ref{proto}.

\section{The setting}\label{setting}
\subsection{The system and its transfer operator}
Let $(I, \mathcal{B}, m)$ be the measure space, where $I:=[0,1]$, $\mathcal B$ is Borel $\sigma$-algebra, and $m$ is the Lebesgue measure on $I$. Let $T:I\rightarrow I$ be piecewise $C^2$ and expanding (see \cite{LY, P} for original references\footnote{In our work, we do not differentiate between maps with finite number of branches \cite{LY} or countable (infinite) number of branches \cite{P}. All that we need is a setting where assumptions {\bf (A1)} and {\bf(A2)} are satisfied. In fact, using these assumptions, this work can be extended to the multidimensional case \cite{Li0} by taking care of the dimension \cite{Li1} and by working with appropriate observables since the space of functions of bounded variations in higher dimension is not contained in $L^{\infty}$.}  and \cite{BG} for a profound background on such systems). The transfer operator (Perron-Frobenius) \cite{Ba} associated with $T$, $P:L^1\rightarrow  L^1$ is defined by duality:  for $f\in L^1$ and $g\in L^{\infty}$
 $$\int_I f\cdot g\circ T dm= \int_I P(f)\cdot gdm.$$
 Moreover, for $f\in L^1$ we have 
 $$Pf(x)=\sum_{y=T^{-1}x}\frac{f(y)}{|T'(y)|}.$$
For $f\in L^1$, we define 
$$Vf=\inf_{\overline f}\{\text{var}\overline f\, : f=\overline f \text{ a.e.}\},$$
where
$$\text{var}\overline f=\sup\{\sum_{i=0}^{l-1}|\overline{f}(x_{i+1})-\overline{f}(x_i)|\, :0=x_0<x_1<\dots<x_l=1\}.$$
We denote by $BV$ the space of functions of bounded variation on $I$ equipped with 
the norm $||\cdot||_{BV}=V(\cdot)+||\cdot||_{1}$. Further, we introduce the mixed operator norm which will play a key role in our approximation:
$$|||P|||=\underset{||f||_{BV}\le 1}{\sup}||Pf||_1.$$
\subsection{Assumptions} We assume\footnote{It is well known that the systems under consideration satisfy a Lasota-Yorke inequality. What we are assuming in {\bf (A1)} is that there is no constant in front of $\alpha$. Such an assumption is satisfied for instance when $\inf_{x}|T'(x)|>2$ or when $T$ is piecewise onto. When the original map $T$ does not satisfy the assumption {\bf (A1)}, one can find an iterate of $T$ where {\bf (A1)} is satisfied, and then apply the results of this paper.}:\\
\noindent {\bf (A1)} $\exists\,\ \alpha\in (0,1)$, and $B_0\ge 0$ such that $\forall f\in BV$
$$VPf\le\alpha Vf+B_0||f||_1;$$ 
\noindent {\bf (A2)} $P$, as operator on $BV$, has $1$ as a simple eigenvalue. Moreover $P$ has no other eigenvalues whose modulus is unity.\\
\begin{remark}
It is important to remark that the constants $\alpha$ and $B_0$ in {\bf (A1)} depend only on the map $T$ and have explicit analytic expressions (see \cite{LY}). 
\end{remark}
The above assumptions imply that $T$ admits a unique absolutely continuous invariant measure $\nu$, such that $\frac{d\nu}{dm}:=h\in BV$. Moreover, the system $(I,\mathcal B, \nu, T)$ is mixing and it enjoys exponential decay of correlations for observables in $BV$ (see \cite{Ba} for a profound background on this topic).  
\subsection{The problem}
Let $\psi\in BV$ and define 
\begin{equation}\label{eq_me}
\sigma^2:=\lim_{n\to\infty}\frac1n\int_{I}\left(\sum_{i=0}^{n-1}\psi(T^ix)-n\int_I\psi d\nu\right)^2d\nu.
\end{equation}
Under our assumptions the limit in \eqref{eq_me} exists (see \cite{HK}), and by using the summability of the correlation decay and the duality property of $P$, one can rewrite $\sigma^2$ as
\begin{equation}\label{eq2}
\sigma^2:=\int_I\ph^2 hdm+2\sum_{i=1}^{\infty}\int_IP^i(\ph h)\ph dm,
\end{equation}
where 
$$\ph:=\psi-\mu\text{ and }\mu:=\int_I\psi d\nu.$$  

The number $\sigma^2$ is called the variance, or the diffusion coefficient, of $\sum_{i=0}^{n-1}\psi(T^ix)$. In particular, for the systems under consideration, it is well known (see \cite{HK}) that the Central Limit Theorem holds: 
$$\frac{1}{\sqrt{n}}\left(\sum_{i=0}^{n-1}\psi(T^ix)-n\int_I\psi d\nu\right){\overset{\text{law}}{\longrightarrow}}\mathcal N(0, \sigma^2).$$
Moreover, $\sigma^2>0$ if and only if $\psi\not= c+\phi\circ T-\phi$, $\phi\in BV$, $c\in \mathbb R$.\\

The goal of this paper is to provide an algorithm whose output approximates $\sigma^2$ with rigorous error bounds. The first step in our approach will be to discretize $P$ as follows: 
\subsection{Ulam's scheme}
 Let $\eta:=\{I_k\}_{k=1}^{d(\eta)}$ be a partition of $[0,1]$ into intervals of size $\lambda(I_k)\le\varepsilon$. 
Let $\mathfrak B_{\eta}$ be the $\sigma$-algebra generated by $\eta$ and for $f\in L^1$ define the projection
  $$\Pi_{\eps}f=E(f|\mathfrak B_{\eta}),$$
and
$$P_{\eps}=\Pi_{\eps}\circ P\circ\Pi_{\eps}.$$
 $P_{\eps}$, which is called Ulam's approximation of $P$, is finite rank operator which can be represented by a (row) stochastic matrix acting on vectors in
$\mathbb R^{d(\eta)}$ by left multiplication. Its entries are given by
$$P_{kj}=\frac{\lambda(I_k\cap T^{-1}(I_j))}{\lambda(I_k)}.$$
The following lemma collects well known results on $P_\eps$. See for instance \cite{Li1} for proofs of (1)-(4) of the lemma, and \cite{Li1, GN} and references therein for statement (5) of the lemma.
\begin{lemma}\label{le0}
For $f\in BV$ we have
\begin{enumerate}
\item $V(\Pi_{\eps}f)\le V(f)$;
\item $||f-\Pi_{\eps}f||_{1}\le\varepsilon V(f)$;
\item  $$VP_{\eps}f\le\alpha Vf+B_0||f||_1,$$ where $\alpha$ and $B_0$ are the same constants that appear in {\bf(A1)}; 
\item $|||P_{\eta}-P|||\le\Gamma\varepsilon,$ where $\Gamma=\max\{\alpha+1,B_0\}$;
\item $P_{\eps}$ has a unique fixed point $h_{\eps}\in BV$. Moreover, $\exists$ a computable constant $K_*$ such that
$$||h_{\eps}-h||_1\le K_*\eps\ln\eps^{-1}.$$
In particular, for any $\tau>0$, there exists $\eps_*$ such that $||h_{\eps_*}-h||_1\le \tau$.
\end{enumerate}
\end{lemma}
\subsection{Statement of the general result}
Define
$$\phe:=\psi-\mu_{\eps}\text{ and }\mu_{\eps}:=\int_{I}\psi h_{\eps}dm.$$
Set 
$$\sigma^2_{\eps,l}:=\int_I\phe^2 h_\eps dm+2\sum_{i=1}^{l-1}\int_IP_\eps^i(\phe h_\eps)\phe dm.$$
\begin{theorem}\label{main}
For any $\tau>0$, $\exists$ $l_*>0$ and $\eps_*>0$ such that
$$|\sigma^2_{\eps_*,l_*}-\sigma^2|\le \tau.$$
\end{theorem}
\begin{remark}
Theorem \ref{main} says that given a pre-specified tolerance on error $\tau>0$, one finds $l_*>0$ and $\eps_*>0$ so that $\sigma^2_{\eps_*,l_*}$ approximates $\sigma$ up to the pre-specified error $\tau$. In subsection \ref{alg} we provide an algorithm that can be implemented on a computer to find $l_*$ and $\eps_*$, and consequently $\sigma^2_{\eps_*,l_*}$.
\end{remark}
To illustrate the issue of the rate of convergence and to elaborate on why we define the approximate diffusion by $\sigma^2_{\eps,l}$ as a truncated sum, let us define
$$\sigma^2_{\eps}:=\int_I\phe^2 h_\eps dm+2\sum_{i=1}^{\infty}\int_IP_\eps^i(\phe h_\eps)\phe dm.$$
\begin{theorem}\label{rate}
$\exists$ a computable constant $\tilde K_*$ such that
$$|\sigma^2_{\eps}-\sigma^2|\le\tilde K_*\eps(\ln\eps^{-1})^2.$$
\end{theorem}
\begin{remark}\label{r1}
Note that $\sigma_{\eps}^2$ can be written as
\begin{equation}\label{res}
\begin{split}
\sigma^2_{\eps}&=\int_I\phe^2 h_\eps dm+2\sum_{i=1}^{\infty}\int_IP_\eps^i(\phe h_\eps)\phe dm\\
&=-\int_I\phe^2 h_\eps+2\int_I\phe({\bf 1}- P_\eps)^{-1}(\phe h_\eps) dm.
\end{split}
\end{equation}
Since $P_\eps$ has a matrix representation, and consequently $(I-P_{\eps})^{-1}$ is a matrix, one may think that $\sigma^2_{\eps}$ provides a more sensible formula to approximate $\sigma^2$ than $\sigma^2_{\eps,l}$. However, from the rigorous computational point of view one has to take into account the errors that arise at the computer level when estimating $(I-P_{\eps})^{-1}$. Indeed $(I-P_\eps)^{-1}$ can be computed rigorously on the computer by estimating it by a finite sum plus an error term coming from estimating the tail of the sum\footnote{Of course, usual computer software would give an estimated matrix of $(I-P_\eps)^{-1}$, but it does not give the errors it made in its approximation.}. This is what we do in Theorem \ref{main}.
\end{remark}
\begin{remark}
 In \cite{BM} an example of a \emph{highly regular} expanding map (piecewise affine) was presented where the \emph{exact} rate of Ulam's method for approximating the invariant density $h$ is $\eps\ln \eps^{-1}$. In Theorem \ref{rate} the rate for approximating $\sigma^2$ is $\eps(\ln\eps^{-1})^2$. This is due to the fact that $||h-h_{\eps}||_1$ is an essential part in estimating $\sigma^2$ and the extra $\ln\eps^{-1}$ appears because of the infinite sum in the formula of $ \sigma^2$.
\end{remark}
\begin{remark}\label{r2} 
By using the representation \eqref{res} of $\sigma_{\eps}^2$, it is obvious that the main task in the proof of Theorem \ref{rate} is to estimate 
$$|||({\bf 1}-P)^{-1}-({\bf 1}-P_{\eps})^{-1}|||_{BV_0\to L^1},$$
where $BV_0=\{f\in BV \text{ s.t.} \int fdm=0\}$. Thus, it would be tempting to use estimate (9) in Theorem 1 of \cite{KL}, which reads:
\begin{equation}\label{E4}
\begin{split}
&|||({\bf 1}-P)^{-1}-({\bf 1}-P_{\eps})^{-1}|||_{BV_0\to L^1}\\
&\hskip 2.5cm \le |||P-P_{\eps}|||_{BV_0\to L^1}^{\theta}(c_1||({\bf 1}-P_{\eps})^{-1}||_{BV_0}+c_2||({\bf 1}-P_{\eps})^{-1}||_{BV_0}^2),
\end{split}
\end{equation}
where $\theta=\frac{\ln(r/\alpha)}{\ln(1/\alpha)}$, $r\in(\alpha,1)$, and $c_1,c_2$ are constants that dependent only on $\alpha$, $B_0$ and $r$. On the one hand, this would lead to a shorter proof than the one we present in section \ref{proofs}; however, estimate \eqref{E4} would lead to a convergence rate of order $\eps^{\theta}$, where $0<\theta<1$ which is slower than the rate obtained in Theorem \ref{rate}. Naturally, this have led us to opt for using the proofs of section \ref{proofs}.
\end{remark}
\subsection{Approximating the diffusion coefficient for non-uniformly expanding maps} We now show that Theorem \ref{main} can be used to approximate the diffusion coefficient for non-uniformly expanding maps. We restrict the presentation to the model that was popularized by Liverani-Saussol-Vaienti \cite{LSV}. Such systems have attracted the attention of both mathematicians \cite{ LSV, Y} and physicists because of their importance in the study of intermittent transition to turbulence \cite{PM}. Let
\begin{equation}\label{eqn_lsv}
S(x)=\begin{cases}
       x(1+2^{\gamma}x^{\gamma}) \quad x\in[0,\frac{1}{2}]\\
       2x-1 \quad \quad \quad x\in(\frac{1}{2},1]
       \end{cases},
\end{equation}
where the parameter $\gamma\in (0, 1)$. $S$ has a neutral fixed point at $x=0$. It is well known that $S$ admits a unique absolutely continuous probability measure $\tilde\nu$, and the system enjoys polynomial decay of correlation for H\"older observables \cite{Y}. For $\gamma\in (0,\frac12)$ it is known that the system satisfies the Central Limit Theorem\footnote{See \cite{Y} for this result and \cite{G} for a more general result.}. To study such systems it is often useful to first induce $S$ on a subset of $I$ where the induced map $T$ is uniformly expanding. In particular for the map \eqref{eqn_lsv}, denoting its first branch by $S_1$ and the second one by $S_2$, one can induce $S$ on $\Delta:=[\frac12,1]$. For $n\ge 0$ we define
$$x_0:=\frac12\text{ and }x_{n+1} = S_{1}^{-1}(x_{n}).$$
Set
$$W_{0}:=(x_0,1),\text{and } W_{n}:=(x_{n},x_{n-1}),\, n\ge 1.$$
For $n\ge 1$, we define
$$Z_n:=S_2^{-1}(W_{n-1}).$$
Then we define the induced map $T:\Delta\to\Delta$ by
\begin{equation}\label{induced}
T(x)=S^{n}(x)\text{ for $x\in Z_{n}$.}
\end{equation}
Observe that
$$S(Z_{n})= W_{n-1}\text{ and } R_{Z_{n}}=n,$$
where $R_{Z_{n}}$ is the first return time of $Z_{n}$ to $\Delta$. For $x\in\Delta$, we denote by $R(x)$ the first return time of $x$ to $\Delta$. Let $f$ be H\"older with $\int_{I} fd\tilde\nu=0$.  Then diffusion coefficient of the system $S$ can be written using the data of the induced map $T$ (see \cite{G}). In particular, for $x\in\Delta$, writing $\psi(x)=\sum_{i=0}^{R(x)-1}f(S^ix)$, the diffusion coefficient is given by
$$\sigma^2:=\int_{\Delta}\psi^2 hdm_{\Delta}+2\sum_{i=1}^{\infty}\int_{\Delta}P^i(\psi h)\psi dm_{\Delta},$$
where $h$ is the unique invariant density of induced map $T$, $P$ is the Perron-Frobenius operator associated with $T$, and $m_{\Delta}$ is normalized Lebesgue measure on $\Delta$. Thus, for $\psi\in BV$ one can use\footnote{Although $T$ has countable (infinite) number of branches, one can still implement the approximation on a computer. One way to do so is as follows: first one may perform an intermediate step by considering a map $\tilde T$ identical to $T$ on $I\setminus H$, such that  $\tilde T$ has finite number of branches on $I\setminus H$ while on $H$ it has, say, one expanding linear branch, with $m(H)\le \delta$ and $\frac{\delta}{\tau}$ is sufficiently small. The diffusion coefficients of $T$ and $\tilde T$ can be made arbitrarily close using the result of \cite{KHK}, and then one can apply Ulam's method and Theorem \ref{main} to $\tilde T$.},  Theorem \ref{main} to approximate $\sigma^2$. 
\section{Proofs and an Algorithm}\label{proofs} 
We first prove two lemmas that will be used to prove Theorem \ref{main}. The explicit estimates of Lemma \ref{le2} below will also be used in Subsection \ref{alg} where we present our algorithm to rigorously estimate diffusion coefficients.
\begin{lemma}\label{le1} For $\psi\in BV$, we have
\begin{enumerate}
\item $||\ph||_{\infty}\le2||\psi||_{\infty}$ and $||\phe||_{\infty}\le2||\psi||_{\infty}$;
\item $|\int_I(\ph^2 h-\phe^2 h_\eps)dm|\le8||\psi||_{\infty}^2||h_\eps-h||_1.$
\end{enumerate}
\end{lemma}
\begin{proof}
Using the definition of $\ph$, $\phe$ we get (1). We now prove (2). We have  
\begin{equation}\label{eq3}
\begin{split}
|\int_I(\ph^2_{\eps}-\ph^2)hdm|=&|\int_I(\ph_{\eps}-\ph)(\ph_{\eps}+\ph)h dm|=|\int_I(\mu-\mu_{\eps})(2\psi-\mu-\mu_{\eps})hdm|\\
&\le4||\psi||_{\infty}|\mu_{\eps}-\mu|\int_Ih dm\le4||\psi||_{\infty}^2||h_\eps-h||_1.
\end{split}
\end{equation}
We now use (1) and (\ref{eq3}) to get
\begin{equation*}
\begin{split}
|\int_I(\ph^2 h-\phe^2 h_\eps)dm|&\le |\int_I(\ph^2 h-\phe^2 h)dm|+|\int_I(\phe^2 h-\phe^2 h_\eps)dm|\\
&\le 8||\psi||_{\infty}^2||h_\eps-h||_1.
\end{split}
\end{equation*}
\end{proof}
\begin{lemma}\label{le2}
For any $l\ge1$ we have
\begin{equation*}
\begin{split}
&|\sum_{i=1}^{l-1}\int_I\left(P^i_{\eps}(\phe h_{\eps})\phe-P^i(\ph h)\ph\right)dm|\le 8(l-1)\cdot ||\psi||^2_{\infty}\cdot ||h_\eps-h||_1\\
&+ 2||\psi||_{\infty} |||P_{\eps}-P||| \sum_{i=1}^{l-1}\sum_{j=0}^{i-1}\left( 2||\psi||_{\infty}(B_j+1+\frac{\alpha^jB_0}{1-\alpha})+\frac{\alpha^j(B_0+1-\alpha)}{1-\alpha}V\psi\right),
\end{split}
\end{equation*}
where $B_{j}=\sum_{k=0}^{j-1}\alpha^{k}B_0.$
\end{lemma}
\begin{proof}
\begin{equation*}
\begin{split}
&|\sum_{i=1}^{l-1}\int_I\left(P^i_{\eps}(\phe h_{\eps})\phe-P^i(\ph h)\ph\right)dm|\\
&\le |\sum_{i=1}^{l-1}\int_I\left(P^i_{\eps}(\phe h_{\eps})\phe-P^i_\eps(\ph h)\ph\right)dm|+|\sum_{i=1}^{l-1}\int_I\left(P^i_{\eps}(\ph h)\ph-P^i(\ph h)\ph\right)dm|\\
&\le|\sum_{i=1}^{l-1}\int_I P^i_{\eps}(\phe h_{\eps}-\ph h)\psi dm|+|\sum_{i=1}^{l-1}\int_I\left(P^i_{\eps}(\phe h_{\eps})\mu_{\eps}-P^i_\eps(\ph h)\mu\right)dm|\\
&\hskip 5cm+|\sum_{i=1}^{l-1}\int_I\left(P^i_{\eps}(\ph h)\ph-P^i(\ph h)\ph\right)dm|\\
&:=(I)+(II) +(III).
\end{split}
\end{equation*}
We have
\begin{equation}\label{es1}
\begin{split}
(I)&\le||\psi||_{\infty}\sum_{i=1}^{l-1}\int_I |\phe h_{\eps}-\ph h|dm\\
&= ||\psi||_{\infty}\cdot(l-1)\int_I |\phe h_{\eps}-\phe h+\phe h-\ph h|dm\\
&\le ||\psi||_{\infty}\cdot( l-1)\left(||\phe||_{\infty}||h_\eps-h||_1+|\mu-\mu_{\eps}|\right)\\
&\le 3||\psi||^2_{\infty}\cdot( l-1) \cdot ||h_\eps-h||_1.
\end{split}
\end{equation}
We know estimate $(II)$:
\begin{equation}\label{es2}
\begin{split}
(II)&\le |\sum_{i=1}^{l-1}\int_I\left(P^i_{\eps}(\phe h_{\eps})\mu_{\eps}-P^i_\eps(\ph h)\mu_\eps\right)dm|+|\sum_{i=1}^{l-1}\int_I\left(P^i_\eps(\ph h)\mu_\eps-P^i_\eps(\ph h)\mu\right)dm|\\
&\le(l-1)|\mu_\eps| \int_I\left|\phe h_{\eps}-\ph h\right|dm +2(l-1)\cdot||\psi||_{\infty}|\mu_{\eps}-\mu|\\
&\le 3||\psi||^2_{\infty}\cdot( l-1) \cdot ||h_\eps-h||_1+2(l-1)\cdot||\psi||^2_{\infty}||h_{\eps}-h||_1\\
&=5||\psi||^2_{\infty}\cdot( l-1) \cdot ||h_\eps-h||_1.
\end{split}
\end{equation}
Finally we estimate $(III)$
\begin{equation}\label{es3}
\begin{split}
(III)&\le2||\psi||_{\infty}\sum_{i=1}^{l-1}\sum_{j=0}^{i-1}||P_{\eps}^{i-1-j}(P_\eps-P)P^j(\ph h)||_1\\
&\le 2||\psi||_{\infty}\cdot |||P_{\eps}-P|||\cdot \sum_{i=1}^{l-1}\sum_{j=0}^{i-1}||P^{j}(\ph h)||_{BV}\\
&\le 2||\psi||_{\infty}\cdot |||P_{\eps}-P|||\cdot \sum_{i=1}^{l-1}\sum_{j=0}^{i-1}\left(\alpha^jV(\ph h)+(B_j+1)||\ph h||_1\right)\\
&\le 2||\psi||_{\infty} |||P_{\eps}-P||| \sum_{i=1}^{l-1}\sum_{j=0}^{i-1}\left( 2||\psi||_{\infty}(B_j+1+\frac{\alpha^jB_0}{1-\alpha})+\frac{\alpha^j(B_0+1-\alpha)}{1-\alpha}V\psi\right),
\end{split}
\end{equation}
where in the above estimate we have used {\bf(A1)} and its consequence that $Vh\le\frac{B_0}{1-\alpha}$. Combining estimates \eqref{es1},\eqref{es2} and \eqref{es3} completes the proof of the lemma.
\end{proof}
\begin{proof} (Proof of Theorem \ref{main}) 
\begin{equation*}
\begin{split}
|\sigma^2_{\eps,l}-\sigma^2|&\le |\int_I(\ph^2 h-\phe^2 h_\eps)dm|+ 2|\sum_{i=1}^{l-1}\int_I\left(P^i_{\eps}(\phe h_{\eps})\phe-P^i(\ph h)\ph\right)dm|\\ 
&\hskip 4cm +4||\psi||_{\infty}\sum_{i=l}^{\infty}||P^i(\ph h)||_{1}\\
&:=(I)+(II)+(III).
\end{split}
\end{equation*}
We start with $(III)$. Since $\int_I\ph hdm=0$, there exists a computable constant $C_*$ and a computable number\footnote{There are many ways to approximate $(III)$. In the implementation in section \ref{example} we follow the work of \cite{GNS} to estimate $(III)$.} $\rho_*$, where $\alpha<\rho_*<1$, such that   
\begin{equation*}
||P^i(\ph h)||_{1}\le||P^i(\ph h)||_{BV}\le ||\ph h||_{BV}C_*\rho_*^{i}\le\left(2||\psi||_{\infty}+V(\psi)\right)\frac{B_0+1-\alpha}{1-\alpha}C_*\rho_*^{i}.
\end{equation*}
Consequently, 
$$(III)\le 4||\psi||_{\infty}\left(2||\psi||_{\infty}+V(\psi)\right)\frac{B_0+1-\alpha}{(1-\alpha)(1-\rho_*)}C_*\rho_*^{l}.$$
Thus, choosing $ l_*$ such that 
\begin{equation}\label{esl}
l_*:=\left\lceil\frac{\log(\tau/2)-\log\left(4||\psi||_{\infty}\left(2||\psi||_{\infty}+V(\psi)\right)\frac{B_0+1-\alpha}{(1-\alpha)(1-\rho_*)}C_*\right)}{\log\rho_*}\right\rceil
\end{equation}
implies $$4||\psi||_{\infty}\sum_{i=l_*}^{\infty}||P^i(\ph h)||_{1}\le\frac{\tau}{2}.$$ 
Fix $l_*$ as in \eqref{esl}. Now using Lemmas \ref{le0}, \ref{le1} and \ref{le2}, we can find $\eps_*$ such that 
$$|\int_I(\ph^2 h-\ph_{\eps_*}^2 h_{\eps_*})dm|+ 2|\sum_{i=1}^{l_*-1}\int_I\left(P^i_{\eps_*}(\phe h_{\eps_*})\phe-P^i(\ph h)\ph\right)dm|\le\frac{\tau}{2}.$$
This completes the proof of the theorem.
\end{proof}
\subsection{Algorithm}\label{alg} Theorem \ref{main} suggests an algorithm as follows. Given $T$ that satisfies {\bf (A1)} and {\bf (A2)} and $\tau>0$ a tolerance on error:
\begin{enumerate}
\item Find $l_*$ such that
$$4||\psi||_{\infty}\sum_{i=l_*}^{\infty}||P^i(\ph h)||_{1}\le\frac{\tau}{2}.$$
\item Fix $l_*$ from (1).
\item Find $\eps_*=\text{mesh}(\eta)$ such that 
\begin{equation*}
\begin{split}
&(16(l_*-1)+8)\cdot ||\psi||^2_{\infty}\cdot ||h_{\eps_*}-h||_1\\
&+4||\psi||_{\infty} \sum_{i=1}^{l_*-1}\sum_{j=0}^{i-1}\left( 2||\psi||_{\infty}(B_j+1+\frac{\alpha^jB_0}{1-\alpha})+\frac{\alpha^j(B_0+1-\alpha)}{1-\alpha}V\psi\right) |||P_{\eps_*}-P|||\le\frac{\tau}{2}.
\end{split}
\end{equation*}
\item Output $\sigma^2_{\eps_*,l_*}:=\int_I\hat\psi_{\eps_*}^2 h_{\eps_*} dm+2\sum_{i=1}^{l_*-1}\int_IP_{\eps_*}^i(\hat\psi_{\eps_*} h_{\eps_*})\hat\psi_{\eps_*} dm$. 
\end{enumerate}
\begin{remark}\label{exp}
Note that the split of $\frac{\tau}{2}$ between items (1) and (2) in Algorithm \ref{alg} to lead to an error of at most $\tau$ can be relaxed in following way. One can compute the error in item (1) to be at most $\frac{\tau}{k}$ and in item (2) to be $\frac{k-1}{k}\tau$ for any integer $k\ge 2$. We exploit this fact in the implementation in section \ref{example}.
\end{remark}
\begin{proof} (Proof of Theorem \ref{rate}) 
\begin{equation*}
\begin{split}
|\sigma^2_{\eps}-\sigma^2|&\le |\int_I(\ph^2 h-\phe^2 h_\eps)dm|+ 2|\sum_{i=1}^{l-1}\int_I\left(P^i_{\eps}(\phe h_{\eps})\phe-P^i(\ph h)\ph\right)dm|\\ 
&\hskip 2cm +4||\psi||_{\infty}\sum_{i=l}^{\infty}||P^i(\ph h)||_{BV}+4||\psi||_{\infty}\sum_{i=l}^{\infty}||P_{\eps}^i(\hat\psi_{\eps} h_\eps)||_{BV}\\
&:=(I)+(II)+(III)+(IV).
\end{split}
\end{equation*}
 We first get an estimate on $(III)$ and $(IV)$. There exists a computable constant $C_*$ and a computable number $\rho_*$, where $\alpha<\rho_*<1$, such that   
$$(III)+(IV)\le 8||\psi||_{\infty}\left(2||\psi||_{\infty}+V(\psi)\right)\frac{B_0+1-\alpha}{(1-\alpha)(1-\rho_*)}C_*\rho_*^{l}.$$
For $(II)$, as in Lemma \ref{le2}, in particular \eqref{es3}, and by using Lemma \ref{le0}, we have
\begin{equation*}
\begin{split}
(II)&\le 4||\psi||_{\infty}\sum_{i=1}^{l-1}\sum_{j=0}^{i-1}||P_{\eps}^{i-1-j}(P_\eps-P)P^j(\ph h)||_1+16(l-1)\cdot ||\psi||^2_{\infty}\cdot ||h_\eps-h||_1\\
&\le 4||\psi||_{\infty}\Gamma\cdot\left( \alpha V(\psi) \frac{B_0+1-\alpha}{1-\alpha}+||\psi||_{\infty}\frac{2B_0+\alpha B_0}{1-\alpha}\right)(l-1)\eps\\
&\hskip 8cm+ K_*16(l-1)\eps\ln\eps^{-1}.
\end{split}
\end{equation*}
For $(I)$ we use Lemmas \ref{le0} and \ref{le1} to obtain
$$(I)\le 8||\psi||_{\infty}^2||h_\eps-h||_1\le 8||\psi||_{\infty}^2K_*\eps\ln\eps^{-1}. $$
Finally, choosing $l=\lceil\frac{\ln\eps}{\ln\rho_*}\rceil$ leads to the rate $\tilde K_*\eps(\ln\eps^{-1})^2$.
\end{proof}
\section{Implementation of the algorithm and estimating the diffusion coefficient for Lanford's map}\label{example}
Let
\begin{equation}  \label{map}
T(x)=2x+\frac{1}{2}x(1-x) \hskip 0.5cm {\text{(mod } 1)}.
\end{equation}
The map defined in \eqref{map} is known as Lanford's map \cite{Lanford}. In this section we let $\psi=x^2$ and compute the diffusion coefficient up to a pre-specified error $\tau=0.035$. A plot of $T$ on $[0,1]$ and an approximation of its invariant density computed through Ulam's approximation are plotted in Figure \ref{fig:dynamic}. 
\begin{figure}[h]
\begin{subfigure}[b]{0.45\textwidth}
    \includegraphics[width=50mm,height=50mm]{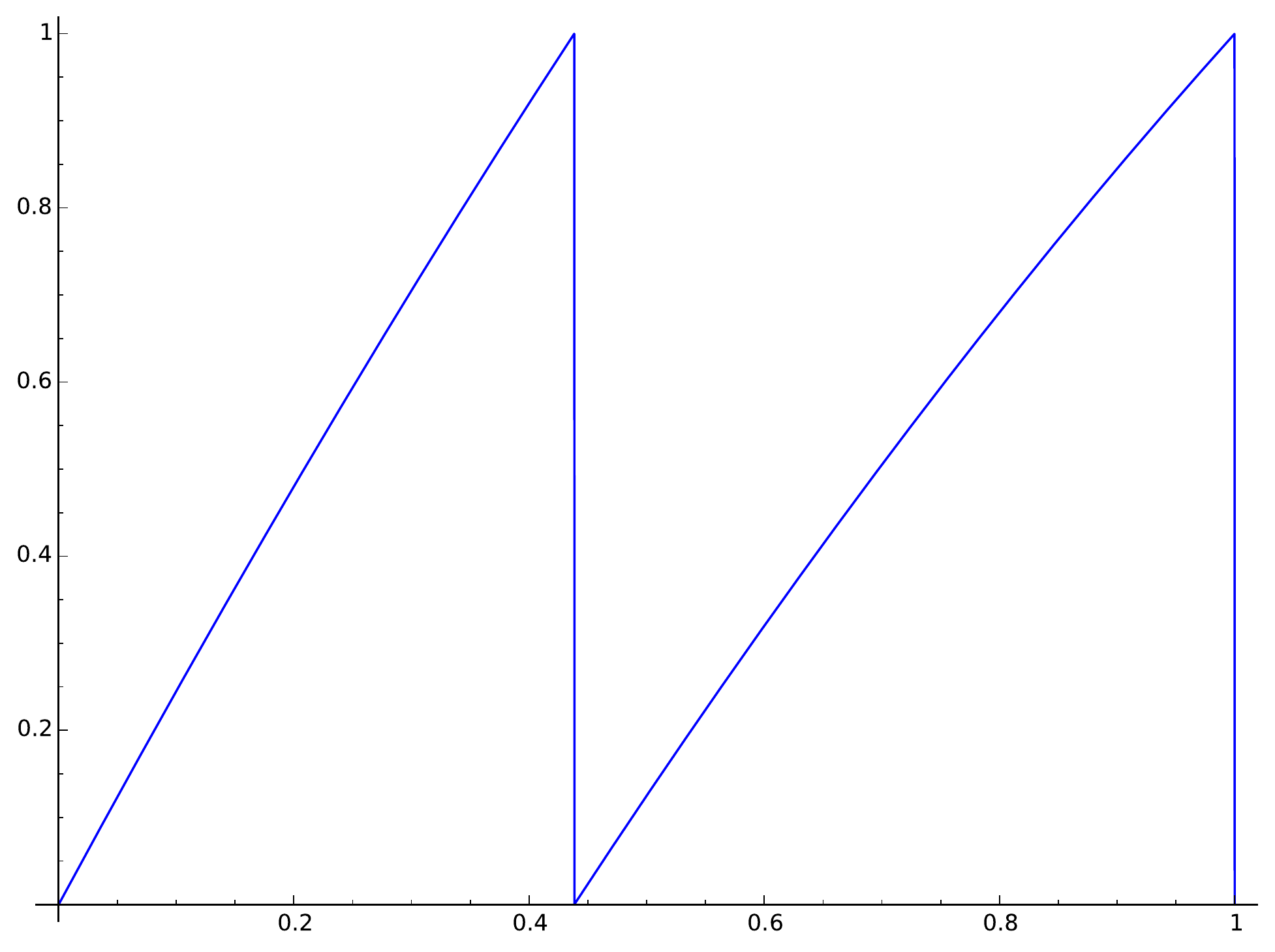}
  \end{subfigure}  
\begin{subfigure}[b]{0.45\textwidth}
    \includegraphics[width=50mm,height=50mm]{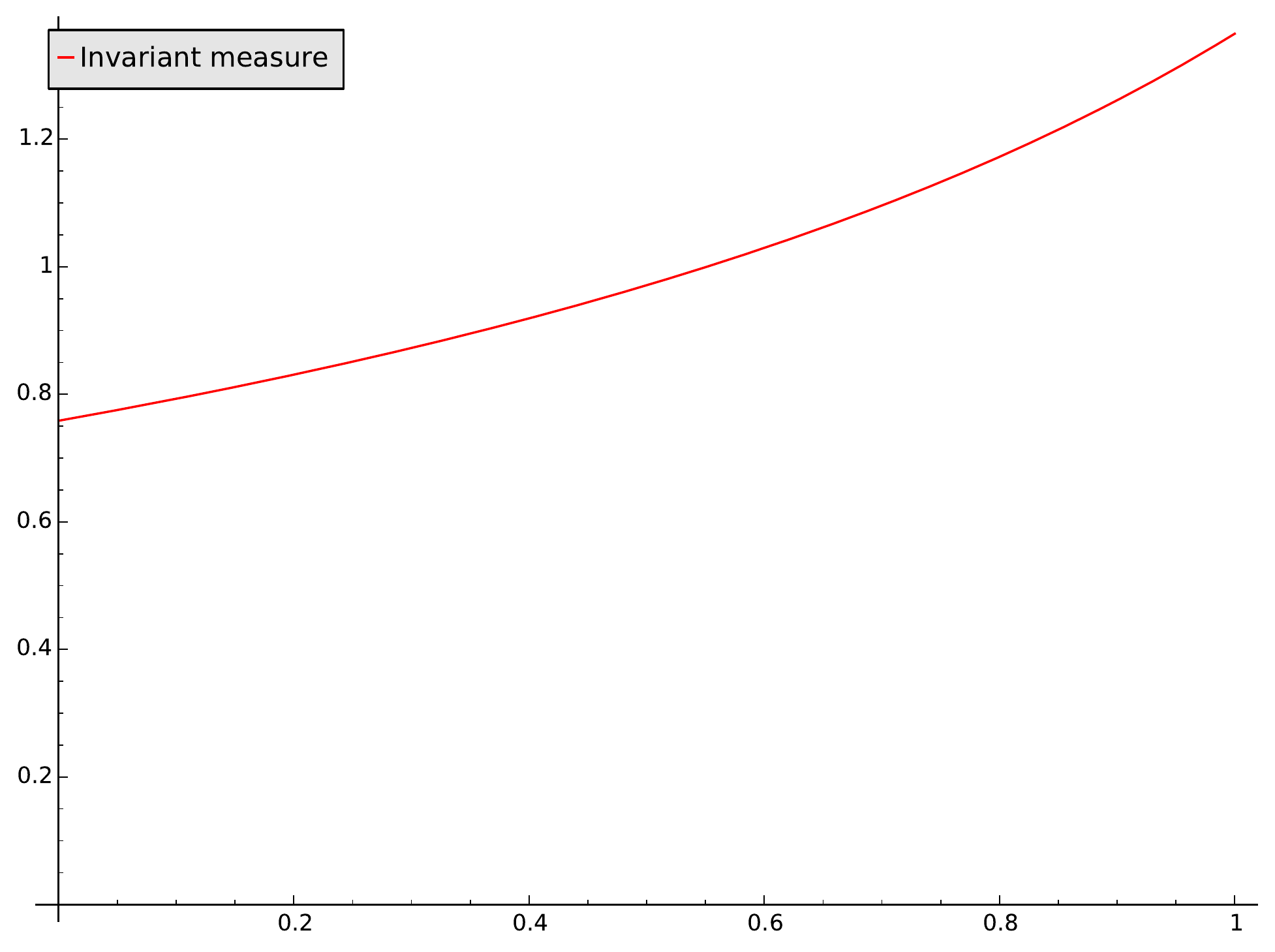}
  \end{subfigure}
\caption{The map $T$ in \eqref{map} to the left, and its approximated invariant density to the right.}
\label{fig:dynamic}
\end{figure}

\bigskip

\subsection{Rigorous projections on the Ulam basis}
To compute the diffusion coefficient rigorously we have to compute rigorously the projection of an observable
on the Ulam basis, i.e., given an observable $\phi$ in $BV$, and the projection $\Pi_{\eps}$
we need to compute rigorously the coefficients $\{v_0,\ldots,v_n\}$ such that
\[\Pi_{\varepsilon}\phi=\sum_{i=0}^{n-1}v_i\cdot \frac{\chi_{I_i}}{m(I_i)},\]
where
\[v_i=\int_{I_i}\phi \,dm.\]

To do so, we will use rigorous integration through interval arithmetics, as explained in the book \cite{Tucker}.

Once an observable is projected on the Ulam basis, many operations involved in the computation of the diffusion
coefficient become componentwise operations on vectors; we explain this point in more details.

The first operation is the integral with respect to Lebesgue measure of an observable projected on the Ulam basis. This is given by the following formula:
\[\int_0^1 \Pi_{\varepsilon}\phi\, dm=\int_0^1\sum_{i=0}^n v_i\frac{\chi_{I_i}}{m(I_i)}\,dm=\sum_i v_i.\]

Suppose now we have computed an approximation $h_{\varepsilon}$ of the invariant density with respect to the partition, i.e.,
$\int_0^1 h_{\varepsilon} dx=1$. In the following we will denote its coefficients on the Ulam basis by $\{w_0,\ldots w_n\}$.
Note that the $i$-th component, $w_i$, is the measure of $I_i$ with respect to the measure $h_{\varepsilon}dm$.

The second operation we are interested in is the pointwise product of functions and the relation of the projection $\Pi_{\eps}$
with this operation. We claim that:
\[\Pi_{\varepsilon} (\phi\cdot h_{\varepsilon})(x)=\Pi_{\varepsilon}\phi(x)\cdot h_{\varepsilon}(x).\]

We will prove this by expressing the components of $\Pi_{\varepsilon} (\phi\cdot h_{\varepsilon})$ as a 
function of the components $\{w_0,\ldots,w_n\}$ of
$h_{\varepsilon}$ and the components $\{v_0,\ldots,v_n\}$ of $\Pi_{\varepsilon}\phi$.
We claim that 
\[
\Pi_{\varepsilon}(\phi\cdot h_{\varepsilon})_i=\frac{v_i\cdot w_i}{m(I_i)}.
\]

First of all recalling that $\chi_{I_i}^2=\chi_{I_i}$ and 
that $\chi_{I_i}\cdot \chi_{I_j}=0$ for $i\neq j$ we have:
\begin{align*}
\sum_i \frac{v_i\cdot w_i}{m(I_i)}\cdot \frac{\chi_{I_i}(x)}{m(I_i)}&=\sum_i v_i \cdot \frac{\chi_{I_i}(x)}{m(I_i)}\sum_i w_j \cdot \frac{\chi_{I_j}(x)}{m(I_j)}=(\Pi_{\varepsilon}\phi)(x)\cdot h_{\varepsilon}(x).
\end{align*}
On the right hand side, since $h_{\varepsilon}$ is constant on each $I_i$ and equal to $w_{i}$, we have:
\[
(\Pi_{\eps} (\phi h_{\varepsilon}))_i=\int_{I_i} h_{\varepsilon}\phi\, dm=\int_{I_i} w_{i}\cdot \frac{\chi_{I_i}}{m(I_i)}\phi\, dm=\frac{w_i}{m(I_i)}\cdot \int_{I_i}\phi\, dm=\frac{w_i\cdot v_i}{m(I_i)}.
\]
These identities simplify the computations when dealing with the Ulam basis. It is worth noting that these identities
imply that:
\[
\int_0^1 \phi\cdot h_{\varepsilon} dm=\sum_i \frac{v_i\cdot w_i}{m(I_i)}.
\]
Moreover, it is worth observing that, if $P_{\varepsilon}$ is the Ulam approximation and $\phi$ is an observable:
\[
P_{\eps}(\phi\cdot h_{\varepsilon})=\Pi_{\eps}P\Pi_{\eps}(\phi\cdot h_{\varepsilon})=\Pi_{\eps}P\Pi_{\eps}\Pi_{\eps}(\phi\cdot h_{\varepsilon})=P_{\eps}(\Pi_{\varepsilon}\phi\cdot h_{\varepsilon}).
\]

\subsection{Item (1) in Algorithm 3.1}
In this step, we find $l^{\ast }$ such that item (1) of Algorithm 3.1 is
satisfied. In particular we want to find $l^{\ast }$ such that 
\begin{equation*}
4||\psi ||_{\infty }\sum_{i=l^{\ast }}^{+\infty }||P^{i}((\hat{\psi}\cdot
h))||_{1}\leq \frac{\tau }{256}.
\end{equation*}%
As explained in Remark \ref{exp}, instead of verifying item (1) to be
smaller than $\frac{\tau }{2}$, we verify that it is smaller than $\frac{%
\tau }{256}$. This will give us more room in verifying item (2) so that the
sum of the errors from both items is smaller than $\tau $. Since the system
satisfies \textbf{(A2)}, there exist $0<\rho _{\ast }<1$, and $C_{\ast }>0$
such that for any $g\in BV_{0}$, and any $k\in \mathbb{N}$, 
\begin{equation}
\Vert P^{k}g\Vert _{1}\leq C_{\ast }\rho _{\ast }^{k}\Vert g\Vert _{BV}.
\label{decay}
\end{equation}%
We want to find a $0<\rho _{\ast }<1$ and a $C_{\ast }>0$ so that %
\eqref{decay} is satisfied.

Once these two numbers are computed, we can easily find $l_{\ast }$ (see %
\eqref{esl}) so that item (1) is satisfied. To compute $\rho ^{\ast }$
and $C_{\ast }$ we follow \cite{GNS} whose main idea is to build a system of
iterated inequalities governed by a positive matrix $\mathcal{M}$ such that: 
\begin{equation}
\left( \!%
\begin{array}{c}
\Vert \mathnormal{P}^{in_{1}}g\Vert _{BV} \\ 
\Vert \mathnormal{P}^{in_{1}}g\Vert _{L^{1}}%
\end{array}%
\!\right) \preceq \mathcal{M}^{i}\left( \!%
\begin{array}{c}
\Vert \mathnormal{g}\Vert _{BV} \\ 
\Vert \mathnormal{g}\Vert _{L^{1}}%
\end{array}%
\!\right) ,  \label{pain}
\end{equation}%
where $\preceq $ means component-wise inequalities, e.g. for vectors $%
\overrightarrow{x}=(x_{1},x_{2})$ and $\overrightarrow{y}=(y_{1},y_{2})$, if 
$\overrightarrow{x}\preceq \overrightarrow{y}$, then, $x_{1}\leq y_{1}$ and $%
x_{2}\leq y_{2}$.

By using Lemma \ref{le0} and Appendix \ref{sec:A}, we get that, if $%
||P^n_{\varepsilon}|_{BV_0}||_1\leq\alpha_2$, the following inequalities are
satisfied: 
\begin{equation}  \label{eq4}
\begin{cases}
\|P^{n_1 }f\|_{BV} \le \alpha^{n_1 } \| f\|_{BV} + (\frac{B_0}{1-\alpha})
\|f\|_1 \\ 
\| P^{n_1 } f \|_1 \le \alpha_2 \| f \|_1 + \varepsilon M( (\frac{1+\alpha}{%
1-\alpha} ) \| f\|_{BV} + B_0{n_1 }(1+\alpha+M) \| f\|_1.%
\end{cases}%
\end{equation}

Using the inequalities above we have that: 
\begin{equation*}
\mathcal{M}= 
\begin{pmatrix}
\alpha^{n_1 } & B \\ 
\varepsilon M(\frac{1+\alpha}{1-\alpha} ) & \varepsilon M B_0{n_1 }%
(1+\alpha+M) + \alpha_2%
\end{pmatrix}%
.
\end{equation*}

Following the ideas of \cite{GNS} we have that 
\begin{equation}  \label{pain1}
\| P^{k n_1} g \|_{1} \le \frac{1}{b} \rho_*^{k} \|g \|_{BV},
\end{equation}
where $\rho_*$ is the dominant eigenvalue of $\mathcal{M}$ and $(a,b)$ is
the corresponding left eigenvector.

Thus, our main task now is to identify all the entries of the above matrix.
The first one is $M$, a bound on the $L^{1}$ norm of the iterates of $P$ and 
$P_{\varepsilon }$; by definition, we have that $||P^{n}||\leq 1$
and $||P_{\varepsilon }||_{1}\leq 1$, therefore $M=1$. The two constants $\alpha _{2}$ and $n_{1}$ in $\mathcal{M}$ are two
constants that give us an estimate of the speed at which $P_{\varepsilon }$
contracts the space $BV_{0}$. Let $P_{\varepsilon }$ be the discretized Ulam
operator and fix $\alpha _{2}<1$; we want to find and $n_{1}\geq 0$ such
that $\forall v\in BV_{0}$ 
\begin{equation}
\Vert P_{\varepsilon }^{n_{1}}v\Vert _{1}\leq \alpha _{2}\Vert v\Vert _{1}
\end{equation}%
with $\alpha _{2}<1$. We follow the idea of \cite{GN} and use the computer
to estimate $n_{1}$ with a rigorous computation; we refer to their paper for
the algorithm used to certify $n_{1}$ and the corresponding numerical
estimates and methods. Consequently, \eqref{pain} is satisfied with $n_1=28$ , $\alpha\leq
0.66666667$, $B\leq 1.444444445$, $\varepsilon=1/16384$, $M=1$, $%
\alpha_2=1/64$; i.e., 
\begin{equation*}
\mathcal{M}= 
\begin{pmatrix}
1.18\cdot 10^{-5} & 4.3333334 \\ 
0.000306 & 0.022208%
\end{pmatrix}%
.
\end{equation*}
Thus, $\rho_*= 0.05$ and the eigenvector $(a,b)$ associated to the
eigenvalue $\rho_*$ is given by $a\in [0.006,0.007]$, $b\in [0.993,0.994]$.

Thus, by \eqref{pain1}, we obtain 
\begin{equation*}
\| P^{28k}g \|_{L^1} \le (1.007) \times 0.05^{k} \|g \|_{BV}
\end{equation*}
Consequently we can compute $l_*\geq 112$.

\begin{remark}
Using equation \eqref{pain1} and
supposing $l_{\ast }=k\cdot n_{1}$ we see that, for any $\psi $ in $BV_{0}$: 
\begin{equation*}
\sum_{i=l_{\ast }}^{+\infty }||P^{i}(\psi )||_{1}\leq ||\psi ||_{BV}\frac{1%
}{b}\cdot n_{1}\sum_{i=k}^{+\infty }\rho _{\ast }^{i}\leq ||\psi ||_{BV}%
\frac{1}{b}n_{1}\frac{\rho _{\ast }^{k}}{1-\rho _{\ast }}.
\end{equation*}
\end{remark}

\subsection{Item (2) of Algorithm \protect\ref{alg}}

From now on $l_*$ is fixed and it is equal to $112$. So far, we executed the
first loop of the Algorithm \ref{alg}; i.e., 
\begin{equation*}
4 \| \psi \|_\infty \sum_{i=112}^\infty \| P^i (\hat{\psi})\|_{1} \le \frac{%
\tau}{256}.
\end{equation*}

\begin{remark}
Note in our computation above we have obtained $l_*$ such that
\[
4||\psi ||_{\infty }\sum_{i=l^{\ast }}^{+\infty }||P^{i}((\hat{\psi}\cdot
h))||_{1}\leq \frac{0.01 }{256}\leq \frac{\tau}{256}.
\]
\end{remark}

\subsection{Item (3) of Algorithm \protect\ref{alg}}

In this step, we have to find $\varepsilon_*$, a mesh size of the Ulam
discretization, such that 
\begin{equation}  \label{pain2}
\begin{split}
&(16(l_*-1)+8) \cdot \| \psi\|_\infty^2 \cdot \|h_{\varepsilon_*}-h\|_1 \\
&+4\| \psi\|_\infty \sum_{i=1}^{l_*-1} \sum_{j=0}^{i-1} \bigg(2\|
\psi\|_\infty (B_j +1 + \frac{\alpha^j B_0}{1-\alpha}) + \frac{\alpha^j (B_0
+1-\alpha)}{1-\alpha} V \psi\bigg) ||| P_{\varepsilon_*}-P||| \\
&\le \frac{255}{256}\tau.
\end{split}%
\end{equation}

To bound this term we need a rigorous approximation of the $T$-invariant
density $h$, in the $L^{1}$-norm; we follow the ideas (and refer to the
algorithm) of \cite{GN}. Set: 
\begin{equation}\label{k1}
\kappa :=4\Vert \psi \Vert _{\infty }|||P_{\varepsilon _{\ast
}}-P|||\sum_{i=1}^{l_{\ast }-1}\sum_{j=0}^{i-1}\bigg(2\Vert \psi \Vert
_{\infty }(B_{j}+1+\frac{\alpha ^{j}B_{0}}{1-\alpha })+\frac{\alpha
^{j}(B_{0}+1-\alpha )}{1-\alpha }V\psi \bigg).
\end{equation}
The following table contains, for different mesh sizes $\varepsilon$, error bounds for the terms in equation \eqref{pain2}; in particular a bound on $\kappa$ defined in \eqref{k1}: 
\begin{center}
\begin{tabular}{l|l|l|l}
$\varepsilon $ & $2^{-12}$ & $2^{-24}$ & $2^{-25}$ \\ \hline
$\Vert h_{\varepsilon _{\ast }}-h\Vert _{1}\leq $ & $0.016$ & $3.2\cdot
10^{-5}$ & $1.7\cdot 10^{-5}$ \\ \hline
$(16(l_{\ast }-1)+8)\cdot \Vert \psi \Vert _{\infty }^{2}\cdot \Vert
h_{\varepsilon _{\ast }}-h\Vert _{1}\leq $ & $28.55$ & $0.0571$ & $0.0304$
\\ \hline
$\kappa \leq $ & $8.08$ & $0.0079$ & $0.00395$.%
\end{tabular}
\end{center}

\subsection{Item (4) in Algorithm \protect\ref{alg}}\label{alg4}
\begin{equation*}
|\sigma _{\varepsilon _{\ast },l_{\ast }}^{2}-\sigma ^{2}|\leq
0.01/256+(0.0304+0.00395)\cdot 255/256\leq 0.0342,
\end{equation*}%
and we compute $\sigma _{\varepsilon _{\ast },l_{\ast }}^{2}$ 
\begin{equation*}
\sigma _{\varepsilon _{\ast },l_{\ast }}^{2}:=\int_{I}\hat{\psi}%
_{\varepsilon _{\ast }}^{2}h_{\varepsilon _{\ast }}dm+2\sum_{i=1}^{l_{\ast
}-1}\int_{I}P_{\varepsilon _{\ast }}^{i}(\hat{\psi}_{\varepsilon _{\ast
}}h_{\varepsilon _{\ast }})\hat{\psi}_{\varepsilon _{\ast }}dm\in \lbrack
0.38,0.381].
\end{equation*}

\begin{remark}
The code implementing rigorous computation of diffusion coefficients for piecewise uniformly expanding maps
is avalaible at the research section of the following personal page:
\begin{center}
\textbf{http://www.im.ufrj.br/nisoli/}
\end{center}
\end{remark}

\subsection{A non rigorous verification}
We also perform a non-rigorous experiment to compute $\sigma^2$ in the above example. Let $\mathcal{F}_\zeta$ be the set of floating point numbers in $[0,1]$ with 
$\zeta$ binary digits.

Note that the system has high entropy, so we have to be careful in our computation and choose $\zeta$ big.
Due to high expansion of the system, in few iterations the ergodic average along the simulated orbit may
have little in common with the orbit of the real system.
So, we have to do computations with a really high number of digits ($\zeta=1024$ binary digits).

Let $\{x_0,\ldots,x_{n-1}\}$ be $n$ random floating points in $\mathcal{F}_l$; fix $k$ and for each $i=0,\ldots, n-1$ let 
\[A_k(x_i)=\frac{1}{k}\sum_{j=0}^{k-1} \phi(T^j(x_i)).\]

Let $\mu$ be an approximation of the average of $\phi$ with respect to the invariant measure,
obtained by integrating the observable using the approximation of the invariant density:
\[
  \mu=[0.383,0.384].
\]

Now, for each point $\{x_0,\ldots,x_{n-1}\}$ we compute the value $A_k(x_0),\ldots,A_k(x_{n-1})$ and from these 
we compute the following two estimators:
\begin{align*}
\tilde{\mu}&=\frac{1}{n}\sum_{i=0}^{n-1} A_k(x_i)\\
\tilde{\sigma}^2 &=\frac{1}{n} \sum_{i=0}^{n-1} \frac{(k\cdot A_k(x_i)-k\mu)^2}{k}.
\end{align*}

The estimator $\tilde{\mu}$ gives a non-rigorous estimate for the average of the observable with respect to the invariant measure,
while the estimator $\tilde{\sigma}^2$ is an estimator for the diffusion coefficient.

\begin{figure}[h]
\begin{subfigure}[b]{0.45\textwidth}
    \includegraphics[width=70mm,keepaspectratio=true]{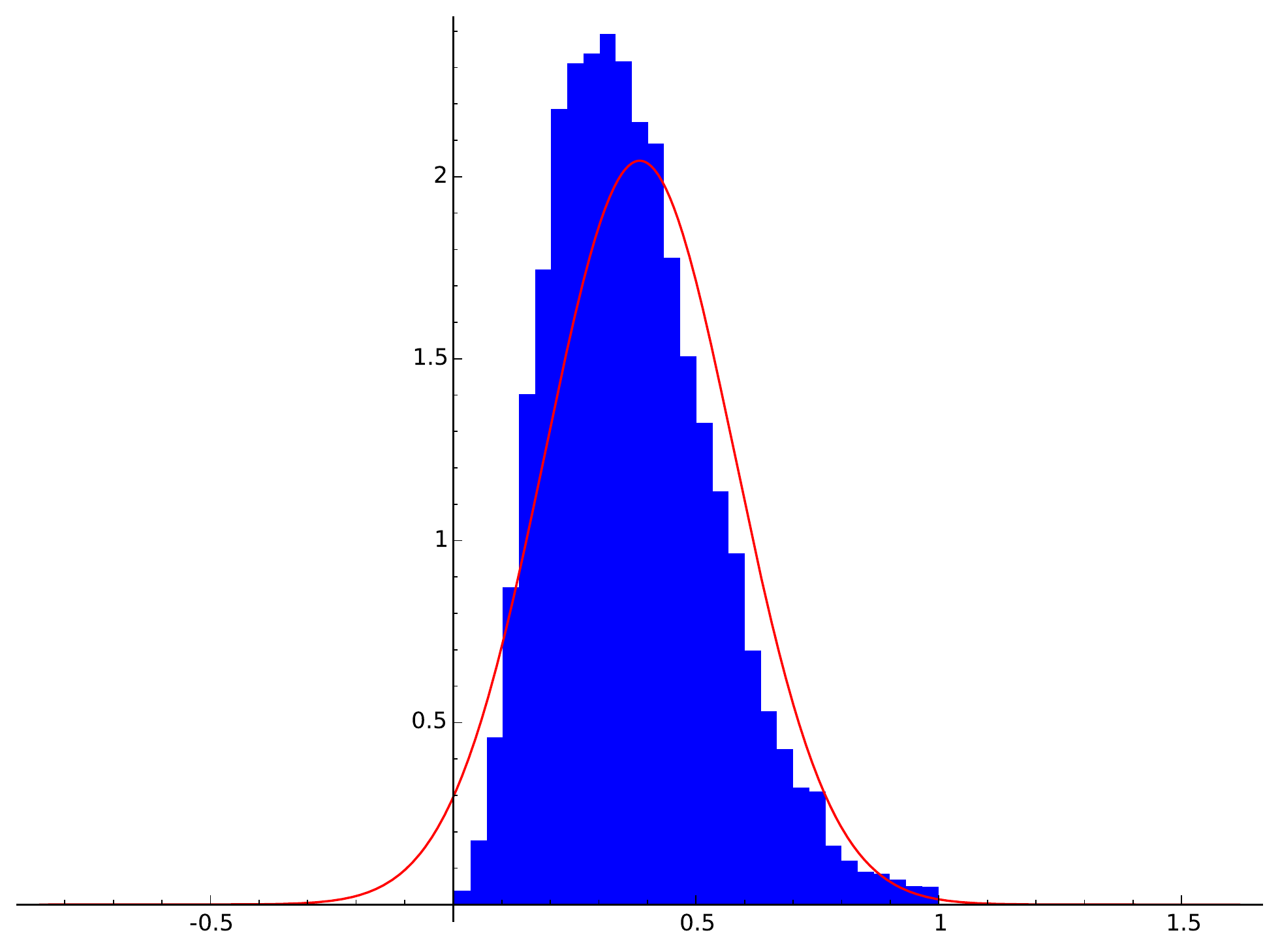}
    \caption{$k=10$}
 
  \end{subfigure}  
\begin{subfigure}[b]{0.45\textwidth}
    \includegraphics[width=70mm,keepaspectratio=true]{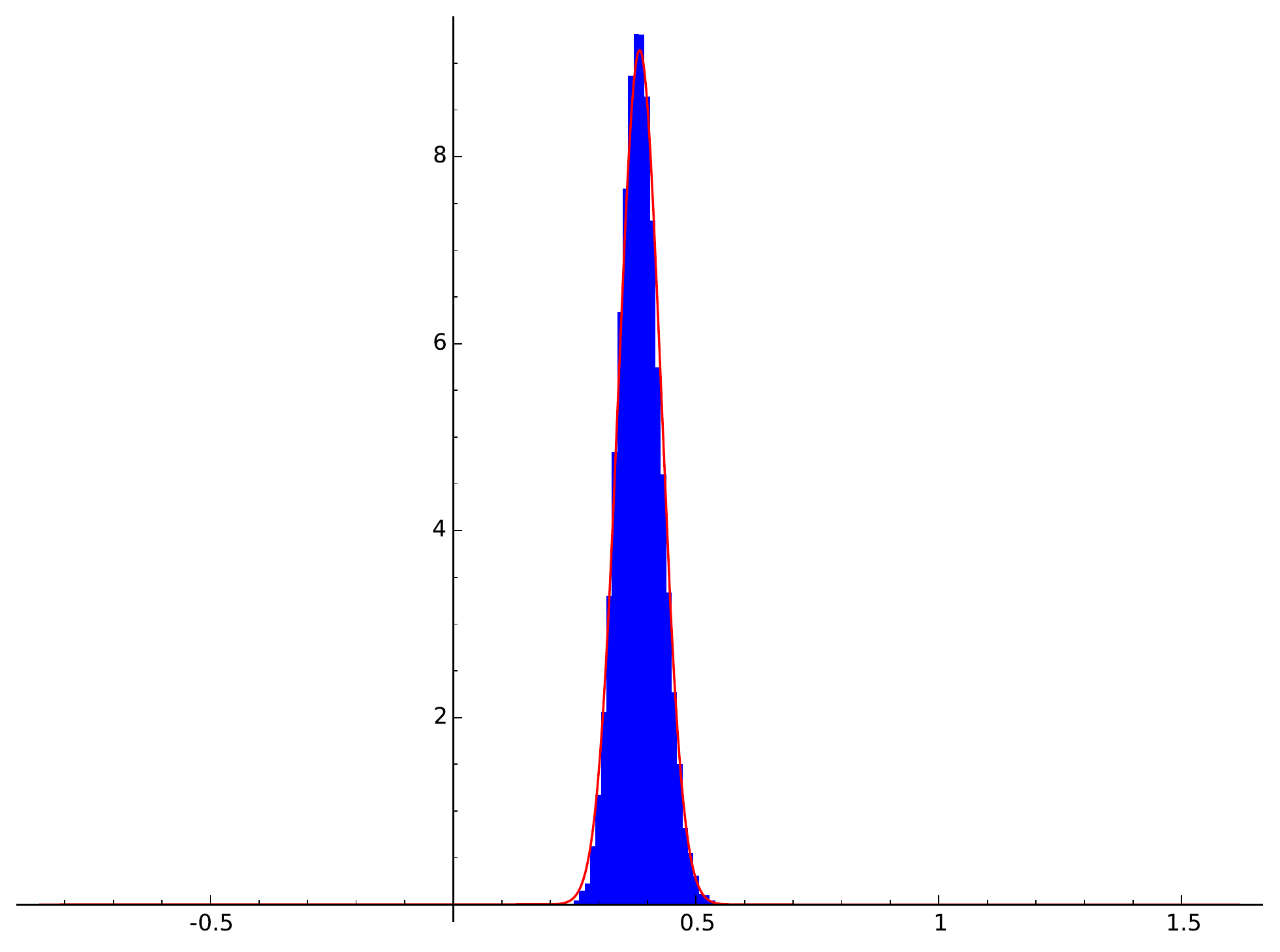}
    \caption{$k=200$}
  \end{subfigure}
\caption{Distribution of the averages $A_{k}(x_i)$, $i=0,\ldots,19999$ for
Lanford's map.}
\label{fig:histogram}
\end{figure}

The table below shows the outcome of the experiment with $n = 20000$. In Figure \ref%
{fig:histogram}, a histogram plot of the distribution of $A_k(x_i)$ for $%
k=10 $, $k = 200$, $n=20000$. In red we have the normal distribution with
average $\mu$ and variance $\sigma_{\varepsilon_*,l_*}^2/\sqrt{k}$.

\begin{center}
\begin{tabular}{l|l|l}
$k$ & $\tilde{\mu}$ & $\tilde{\sigma}^2$\\
\hline
$90$ & $[0.383,0.384]$ & $[0.361,0.362]$ \\
$95$ & $[0.383,0.384]$ & $[0.362,0.363]$ \\
$100$ & $[0.383,0.384]$ & $[0.362,0.363]$ 
\end{tabular}
\end{center}

The output of this non-rigourous experiment is in line with the output from
our rigorous computation in subsection \ref{alg4}.

\appendix
   \section{Proof of  equation \ref{eq4}}
   \label{sec:A}
\begin{lemma}\ \\
\[
  \| (P^n- P^n_\eps) f \|_1  \le  \eps  ( (\frac{1+\alpha}{1-\alpha} ) \| f\|_{BV} + B_0n(2+\alpha)  \| f\|_1.
\]
\end{lemma}
\begin{proof}
  \[
  \| \Pi_\eps \|_1 = \frac{\|\frac{1}{\lambda(I_k)} \int_{I_k} f d \lambda\|_1}{\|f\|_1}\le1.
  \]
  \[
  \| P^n\|_1 =\|P_{\eps}^n\|=1.
  \]
  \[
  \| (P- P_\eps) f \|_1 \le \| \Pi_\eps  P  \Pi_\eps f  - \Pi_\eps  Pf \|_1 + \|  \Pi_\eps  Pf -Pf\|_1=\|\Pi_\eps  P( \Pi_\eps f-f)\|_1 + \|  \Pi_\eps  Pf -Pf\|_1.
  \]
  \[
  \|\Pi_\eps  P( \Pi_\eps f-f)\|_1 \le  \| \Pi_\eps f-f\|_1 \le  \eps V(f)\le \eps\|f\|_{BV};
  \]
  \[
   \|  \Pi_\eps  Pf -Pf\|_1 \le \eps \|Pf\|_{BV} \le \eps ( \alpha \|f\|_{BV} + B_0\|f\|_1).
  \]
  \[
    \| (P- P_\eps) f \|_1 \le  \eps\|f\|_{BV} + \eps ( \alpha \|f\|_{BV} + B_0\|f\|_1)\le   \eps ((1+\alpha) \|f\|_{BV}  + B_0\|f\|_1).
  \]
   \[
    \| (P^n- P^n_\eps) f \|_1 \le \sum_{k=1}^n \| P^{n-k}_\eps(P- P_\eps) P^{k-1} f \|_1\le \| (P- P_\eps) P^{k-1} f \|_1
  \]
  \[
  \le  \eps   \sum_{k=1}^n ((1+\alpha) \|P^{k-1} f\|_{BV}  + B_0\|P^{k-1} f\|_1)
  \]
  \[
  \le  \eps   \sum_{k=1}^n ((1+\alpha) ( \alpha^{k-1} \| f\|_{BV} + (\frac{B_0}{1-\alpha})  \|f\|_1)  + B_0\| f\|_1)
  \]
  \[
  \le  \eps  ( (\frac{1+\alpha}{1-\alpha} ) \| f\|_{BV} + B_0n(2+\alpha)  \| f\|_1.
  \]
  \end{proof}
 \bibliographystyle{amsplain}

\end{document}